\documentclass{amsart}

\usepackage{amsmath}
\usepackage{amssymb}
\usepackage{amsfonts}
\usepackage{enumerate}
\usepackage{vmargin}

\newcommand{\bB}{{\mathbb{B}}}

\newcommand{\bM}{{\mathbb{M}}}
\newcommand{\bN}{{\mathbb{N}}}

\newcommand{\bR}{{\mathbb{R}}}

  \newcommand{\E}{{\mathcal{E}}}

  \newcommand{\M}{{\mathcal{M}}}

  \newcommand{\R}{{\mathcal{R}}}

\renewcommand{\phi}{\varphi}
\newcommand{\upchi}{{\raise.35ex\hbox{\ensuremath{\chi}}}}

\renewcommand{\leq}{\leqslant}
\renewcommand{\geq}{\geqslant}
\renewcommand{\le}{\leqslant}

\newtheorem{thm}{Theorem}[section]

\newtheorem{cor}[thm]{Corollary}
\newtheorem{lemma}[thm]{Lemma}
\newtheorem{rk}[thm]{Remark}

\begin{document}

\title{H\"older estimates for the noncommutative Mazur maps}
\date{}
\author[\'E. Ricard]{\'Eric Ricard}
\address{Laboratoire de Math{\'e}matiques Nicolas Oresme,
Universit{\'e} de Caen Basse-Normandie,
14032 Caen Cedex, France}
\email{eric.ricard@unicaen.fr}

\thanks{{\it 2010 Mathematics Subject Classification:} 46L51; 47A30.} 
\thanks{{\it Key words:} Noncommutative $L_p$-spaces, Mazur maps}

\maketitle

\begin{abstract}
For any von Neumann algebra $\M$, the noncommutative Mazur map
$M_{p,q}$ from $L_p(\M)$ to $L_q(\M)$ with $1\leq p,q<\infty$ is
defined by $f\mapsto f|f|^{\frac {p-q}q}$. In analogy with the
commutative case, we gather estimates showing that $M_{p,q}$ is
$\min\{\frac pq,1\}$-H\"older on balls.
\end{abstract}

\section{Introduction}

In the integration theory, the Mazur map $M_{p,q}$ from $L_p(\Omega)$ to
$L_q(\Omega)$ is defined by $f\mapsto f|f|^{\frac {p-q}q}$. It is an
easy exercise to check that it is $\min\{\frac
pq,1\}$-H\"older. Theses maps also make sense in the noncommutative
$L_p$-setting for which one should expect a similar behavior.  We refer to \cite{PX} for the definitions of $L_p$-spaces for
semifinite von Neumann algebras or more general ones.  Having a
quantitative result on Mazur maps may be useful when dealing with the
structure of noncommutative $L_p$-spaces (see also \cite{Ray}). By
the way, these maps are used implicitly in the definition of $L_p$.
It is known that $M_{p,q}$ is locally uniformly continuous in full
generality (Lemma 3.2 in \cite{Ray}). The lack of references for
quantitative estimates motivates this note.  When dealing with the
Schatten classes (when $\M=B(\ell_2)$), some can be found in
\cite{AP}, more precisely $M_{p,q}$ is $\frac pq$-H\"older when
$1<p<q$. The techniques developed there can be adapted to semifinite
von Neumann algebras but can't reach the case $p=1$. An estimate when 
$q=p'$ and $1<p<\infty$ can also be found in \cite{CL}. Here we aim to
give to the best possible estimates especially for $p=1$.

\medskip
\noindent \textbf{Theorem} \textit{Let $\M$ be a von Neumann algebra,
  for $1\leq p,q<\infty$, $M_{p,q}$ is $\min\{\frac pq,1\}$-H\"older
  on the unit ball of $L_p(\M)$.}

\medskip

The proofs provide a strange behaviour of the H\"older constants
$c_{p,q}$ as $c_{p,q}\to \infty$ if $p<q\to 1$. This reflects the fact that the absolute value is not Lipschitz on $L_1$ or $L_\infty$ but the result may hold with an absolute constant.

We follow a basic approach, showing first the results for semifinite
von Neumann algebras in section 2. We start by looking at positive
elements and then use some commutator or anticommutator estimates.
The ideas here are inspired by \cite{Bha, Dav}. In section 3, we
explain briefly how the Haagerup reduction technique from \cite{HJX}
can be used to get the theorem in full generality.

\section{Semifinite case}

In this section $\M$ is assumed to be semifinite with a nsf trace
$\tau$. We refer to \cite{PX} for definitions. We denote by $L_0(\M,\tau)$ the set of $\tau$-measurable operators, and 
$$L_p(\M,\tau)= \Big\{ f\in L_0(\M,\tau)\;|\; \|f\|_p^p = \tau \big(|f|^p\big)<\infty\Big\}.$$
We drop the reference to $\tau$ in this section.

First we focus on the Mazur maps for positive elements using some
basic inequalities. The first one can be found in \cite{CPPR} Lemma
1.2. An alternative proof can be obtained by adapting the arguments of
\cite{Bha} Theorem X.1.1 to semifinite von Neumann algebras.

\begin{lemma}\label{hardps} 
If $p\geq 1$, $0 < \theta \le 1$, for any $x,\,y\in L_{\theta p}^+(\M)$, we have $$\big\|x^\theta-y^\theta\big\|_{p}\leq \big\|x-y\big\|_{\theta p}^\theta.$$
\end{lemma}

 Its proof relies on the fact that $s\mapsto s^\theta$ is operator monotone and has an integral representation 
$$s^\theta = c_\theta \int_{\bR_+} \frac{t^\theta s}{s+t} \, \frac{dt}{t} \qquad \mbox{with} \qquad c_\theta = \Big( \int_{\bR_+} \frac{u^\theta}{u(1+u)} \, du \Big)^{-1}.$$
\begin{lemma}\label{hardps2} 
If $p\geq 1$, $0 < \theta \le 1$, for any
$x,\,y\in L_{(1+\theta)p}^+(\M)$, we have : 
$$\big\|x^{1+\theta}-y^{1+\theta}\big\|_{p}\leq 3\big\|x-y\big\|_{(1+\theta) p}
 \max\Big\{\big\|x \big\|_{(1+\theta) p},\,\big\|y \big\|_{(1+\theta) p}\Big\}^\theta .$$
\end{lemma}

\begin{proof}
By standard arguments, cutting $x$ and $y$ by some of their spectral
projections, we may assume that $\tau$ is finite $x$ and $y$ are
bounded and invertible to avoid differentiability issues. We use
$$s^{1+\theta} = c_\theta \int_{\bR_+} \frac{t^\theta s^2}{s+t} \,
\frac{dt}{t}.$$ On bounded and invertible elements the maps $f_t:s\mapsto
\frac{s^2}{s+t}= s (s+t)^{-1}s$ are differentiable and 
$$D_sf_t(\delta)= \delta (s+t)^{-1}s+s(s+t)^{-1}\delta-s(s+t)^{-1}\delta (s+t)^{-1}s.$$
Hence putting $\delta=x-y$, we get the integral representation
 $$x^{1+\theta}-y^{1+\theta}= c_\theta \int_0^1 \int_{\bR_+} t^\theta 
D_{y+u\delta} f_t(\delta) \,\frac{dt}{t} {du}.$$
We get, letting $g_t(s)=s(s+t)^{-1}$
 $$x^{1+\theta}-y^{1+\theta}=\int_0^1 \Big((y+u\delta)^\theta \delta +\delta 
(y+u\delta)^\theta\Big) \,du-
c_\theta \int_0^1 \int_{\bR_+} t^\theta g_t(y+u\delta) \delta g_t(y+u\delta)\,\frac{dt}{t} du.$$
The first term is easily handled by the H\"older inequality.
 When $u$ is fixed, note that $g_t(y+u\delta)$ is an invertible positive contraction. Put 
$$\gamma^2=c_\theta\int_{\bR_+} t^\theta g_t(y+u\delta)^2  \frac{dt}{t}\leq (y+u\delta+t)^\theta,$$
and write $g_t(y+u\delta)=v_t\gamma$ so that 
$v_t$ and $y+u\delta$ commute and 
$$c_\theta\int_{\bR_+} t^\theta v_t^2 \frac{dt}{t}=1.$$ Therefore the
 map defined on $\M$, $x\mapsto c_\theta\int_{\bR_+} t^\theta
 v_t x v_t \frac{dt}{t}=1$ is unital completely
 positive and trace preserving, hence it extends to a contraction on
  $L_q$ when  $1\leq q\leq \infty$ (see \cite{HJX} for instance).
Applying it to $x= \gamma \delta \gamma$,
we deduce
$$ \Big\|c_\theta \int_{\bR_+} t^\theta g_t(y+u\delta) \delta g_t(y+u\delta)\,\frac{dt}{t}\Big\|_p \leq \big\| \gamma \delta \gamma\big\|_p\leq \big\|\delta\big
\|_{(1+\theta)p}. \big\|\gamma\big\|_{\frac{2(1+\theta)p}\theta}^2\leq  \big\|\delta\big
\|_{(1+\theta)p}.\big\|y+u\delta\big\|_{(1+\theta)p}^\theta.
$$
thanks to the H\"older inequality again, this is enough to get the conclusion.
\end{proof}

\begin{cor}\label{alpha}
Let $\alpha>1$, $p\geq 1$, for any
$x,\,y\in L_{\alpha p}^+(\M)$:
 $$\big\|x^{\alpha}-y^{\alpha}\big\|_{p}\leq 3\alpha \big\|x-y\big\|_{\alpha p}
 \max\Big\{\big\|x \big\|_{\alpha p},\,\big\|y \big\|_{\alpha p}\Big\}^{\alpha-1} .$$
\end{cor}
\begin{proof}
When $\alpha=n\in \bN$, the result is obvious with constant $n$. For
the general case, put $n=[\alpha]$, so that $\alpha=n(1+\delta)$ with
$0\leq \delta<1$, then use the result for $n$ and then Lemma
\ref{hardps2}.
\end{proof}

Coming back to the Mazur map $M_{p,q}$, Corollary \ref{alpha} says
that $M_{p,q}$ is Lipschitz on the positive unit ball of $L_p(M)$ if
$q<p$. On the other hand Lemma \ref{hardps} says that it is $\frac p
q$-H\"older if $q>p$.  To release the positivity assumption, we will need
a couple of Lemmas but we start by reducing the problem
 to selfadjoint elements by a well known $2\times2$-trick. .

If $x,\, y\in L_p(\M)$ are in the unit ball 
with polar decompositions 
$x=u|x|$ and $y=v|y|$, we want to prove that with $\theta= \min\{\frac pq ,1\}$
\begin{equation}\label{est} \Big\| u |x|^{\frac pq} - v |y|^{\frac pq}\Big\|_q \leq c_{p,q}
\Big\| x-y\Big\|_p^\theta\end{equation}
 In $\bM_2(\M)$ 
equipped with the tensor trace, let
 $$\tilde x=\begin{pmatrix}
 0&x\\
 x^*&0
 \end{pmatrix}\quad\textrm{and}\quad \tilde y=\begin{pmatrix}
 0&y\\
 y^*&0
 \end{pmatrix}\,.$$
They are selfadjoint with  polar decompositions
$$\tilde x=\tilde u |\tilde x|=\begin{pmatrix}
 0&u\\
 u^*&0
 \end{pmatrix}. \begin{pmatrix}
u|x|u^*& 0\\
 0 & |x|
 \end{pmatrix}\quad\textrm{and}\quad 
\tilde y=\tilde v |\tilde y|=\begin{pmatrix}
 0&v\\ v^*&0
 \end{pmatrix}. \begin{pmatrix}
v|y|v^*& 0\\
 0 & |y|
 \end{pmatrix}.$$
The estimates for $\tilde x$ and $\tilde y$ implies that for $x$ and $y$ as 
$$\tilde u |\tilde x|^{\frac pq}=\begin{pmatrix}
 0&u |x|^{\frac pq}\\
 |x|^{\frac pq}u^*&0
 \end{pmatrix}\quad\textrm{and}\quad \tilde v |\tilde y|^{\frac pq}=\begin{pmatrix}
 0&v |y|^{\frac pq}\\
 |y|^{\frac pq}v^*&0
 \end{pmatrix},$$ we have
$$\Big\| \tilde x-\tilde y\Big\|_p=2^{\frac 1p}\Big\| x-y\Big\|_p 
\qquad 
\Big\| \tilde u |\tilde x|^{\frac pq} - \tilde v |\tilde y|^{\frac pq}\Big\|_q
=2^{\frac 1q}\Big\| u |x|^{\frac pq} - v |y|^{\frac pq}\Big\|_q.
$$

Next, we reduce the theorem to a commutator estimate by using the $2\times2$-trick again. We use the commutator notation $[x,b]=xb-bx$.
Put 
$$\tilde x=\begin{pmatrix}
 x&0\\
 0&y
 \end{pmatrix}\quad\textrm{and}\quad \tilde b=\begin{pmatrix}
 0&1\\
 0&0
 \end{pmatrix}\,.$$
So that
$$ \big\| [M_{p,q} (\tilde x),\tilde b] \big\|_{q}= \big\| M_{p,q}(x)-M_{p,q}(y)\big\|_q \qquad\textrm{and}\qquad \big\| [\tilde x,\tilde b] \big\|_{p}=\big\| x-y\big\|_p.$$

\begin{lemma}\label{com} 
If $p\geq 1$, $0 < \theta \le 1$ and $x\in L_{ p}^+(\M)$ and $b\in \M$ then
$$\Big\| \big[x^\theta,b\big]\Big\|_{\frac p \theta} \leq 2^{\theta}
\big\|b\big\|_\infty ^{1-\theta}\big\| [x, b] \big\|_{p}^\theta .$$
$$ \big\| [x, b] \big\|_{p} \leq \frac {12} \theta 
\big\|x\big\|_p ^{1-\theta}\Big\| \big[x^\theta,b\big]\Big\|_{\frac p \theta}.$$
\end{lemma}

\begin{proof}
We start by the first inequality. We may assume $\|b\|_\infty=1$ by homogeneity. Using the $2\times 2$-trick with 
$$\tilde x=\begin{pmatrix}
 x& 0\\
 0& x
 \end{pmatrix}\quad\textrm{and}\quad \tilde b=\begin{pmatrix}
 0&b\\
 b^*&0
 \end{pmatrix}\,,$$
we may assume $b=b^*$ (without loosing on the constant).

Next, as $b=b^*$, we may use the Cayley transform defined by
$$u=(b-i)(b+i)^{-1} ,\qquad b=2i (1-u)^{-1}-i.$$
Clearly $u$ is unitary and functional calculus gives that $\|(1-u)^{-1}\|_\infty\leq \frac 1 {\sqrt 2}$. We have, using Lemma \ref{hardps}
\begin{eqnarray*}
\big\| [x^\theta, b] \big\|_{\frac p \theta} &\leq& 2 
\big\| x^\theta (1-u)^{-1} -(1-u)^{-1} x^\theta \big\|_{\frac p \theta}\\
& \leq & 2 \big\|(1-u)^{-1}\big\|_\infty^2 \big\| x^\theta (1-u) -(1-u) x^\theta \big\|_{\frac p \theta}\\ & \leq& \big\| u^*x^\theta u - x^\theta \big\|_{\frac p \theta}\\
& \leq &  \big\| x u - ux\big\|_{p}^\theta\\
&\leq& \big\|(b+i)^{-1}\big\|_\infty^{2\theta}  \big\| (b+i)x(b-i)  - (b-i)x(b+i)\big\|_{p}^\theta\\
&\leq &2^{\theta}\, \big\| x b -b x \big\|_{p}^\theta.
\end{eqnarray*} 
For the second one, we proceed similarly using Lemma \ref{alpha}.
\end{proof}

\begin{lemma}\label{sum} 
If $p\geq 1$, $0 < \theta \le 1$, there are  constant $C$ and $C_t$ ($t>1$) so that for any  $x,\,y\in L_{ p}^+(\M)$ and $b\in \M$ then
$$\Big\| x^\theta b +b y^\theta \Big\|_{\frac p \theta} \leq C_{\frac p \theta}
\big\|b\big\|_\infty ^{1-\theta}\big\| x b +b y \big\|_{p}^\theta .$$
$$\big\| x b +b y \big\|_{p} \leq C
\big\|x\big\|_p ^{1-\theta}\Big\| x^\theta b +b y^\theta \Big\|_{\frac p \theta} .$$
\end{lemma}

\begin{proof}
Using the  $2\times 2$-trick, we may assume $x=y$. Moreover we may assume that
 $\M$ is finite and $x$ is in $\M$ and invertible. 
Indeed, let $e_n=1_{(\frac 1n,n)}(x)$ and $e_n^\bot=1-e_n$:
$$\big\| x b +b x \big\|_{p} \sim \big\| x e_nbe_n +e_nbe_n x \big\|_{p} +
\big\| e_n x b e_n^\bot \big\|_{p}+ \big\|  e_n^\bot bxe_n \big\|_{p}+\big\| e_n^\bot (xb+bx) e_n^\bot \big\|_{p}$$
$$\big\| x^\theta b +b x^\theta \big\|_{\frac p\theta} \sim \big\|
x^\theta e_nbe_n +e_nbe_n x^\theta \big\|_{\frac p\theta} + \big\| e_nx^\theta
be_n^\bot \big\|_{\frac p\theta}+ \big\| e_n^\bot bx^\theta e_n \big\|_{\frac
  p\theta}+ \big\| e_n^\bot (x^\theta
b+bx^\theta)  e_n^\bot \big\|_{\frac p\theta}.$$ 
If we apply the result in $e_n\M e_n$ where $xe_n\in e_n\M e_n$ is invertible, we get control for the first terms. For the 2 middle terms this is clear by  interpolation as $\big\|
e_nx^\theta be_n^\bot \big\|_{\frac p\theta}\leq \big\| e_nx be_n^\bot
\big\|_{p}^\theta \|b\|_\infty^{1-\theta}$ and $\big\| e_nx be_n^\bot
\big\|_{p}\leq \big\|
e_nx^\theta be_n^\bot \big\|_{\frac p\theta} \|e_nx\|_p^{1-\theta}$. And finally, the last two terms go to 0 with $n\to \infty$.

We will use techniques from \cite{RX} based on Schur multipliers
estimates and interpolation. We use $M_{cb}$ for the completely
bounded norm of a Schur multiplier on $\bB(\ell_2)$. By an obvious
approximation, we may also assume that $x$ has a finite spectrum.  Let
$(\lambda_i)_{i=1...n}$ be the spectrum of $x$ with associated
projections $(p_i)_{i=1...n}$.
We start by the second inequality. For any $\alpha\in[0,1]$, the matrix $\Big( \frac
{\lambda_i^\alpha\lambda_j^{1-\alpha}+\lambda_i^{1-\alpha}
  \lambda_j^\alpha}{\lambda_i+\lambda_j}\Big)_{i,j}$ defines a unital completely positive Schur multiplier
on $\bB(\ell_2^n)$, see the computation in Corollary 2.5 in \cite{RX}. As above, this implies that 
$$\Big\| x^{1-\alpha} b x^\alpha + x^\alpha b
x^{1-\alpha}\Big\|_p \leq  \Big\| xb+bx\Big\|_p.$$
We use$$ x b+bx = x^{1-\theta} (x^\theta b+bx^\theta) + (x^\theta b+bx^\theta)x^{1-\theta} - 
(x^{1-\theta} bx^\theta + x^\theta bx^{1-\theta}).$$
Assume $\theta\geq \frac 13$, by the H\"older inequality
$$\Big\|x b+bx\Big\|_p \leq\big\|x\big\|_p^{1-\theta} \Big( 2\Big\| x^\theta b+bx^\theta\Big\|_{\frac p \theta} + \Big\| x^{\frac {1-\theta}2} b x^{\frac {3\theta-1}2} +x^{\frac {3\theta-1}2}bx^{\frac {1-\theta}2}\Big\|_{\frac p \theta} \Big) $$
Using the above argument with $\alpha=\frac {1-\theta}2$:
$$\Big\|x b+bx\Big\|_p\leq C \big\|x\big\|_p^{1-\theta} \Big\| x^\theta b+bx^\theta\Big\|_{\frac p \theta}.$$
When $\theta <\frac 1 3$, we use 
$$\Big\| x^{1-\theta} bx^\theta + x^\theta bx^{1-\theta}\Big\|_p \leq 2\big\|x\big\|_p^{1-\theta}  \Big\|x^{\frac \theta 2} bx^{\frac \theta 2}\Big\|_{\frac p {\theta}}.$$
And one corrects with a Schur multiplier of the form $\Big( \frac
{\sqrt{\mu_i\mu_j}}  {\mu_i+\mu_j}\Big)_{i,j}$ which has norm 1 (see \cite{RX}) to get 
$$\Big\| x^{1-\theta} bx^\theta + x^\theta bx^{1-\theta}\Big\|_p \leq 2\big\|x\big\|_p^{1-\theta} \Big\| x^\theta b+bx^\theta\Big\|_{\frac p \theta}.$$

For the first inequality, the result is then a particular case
of the main theorem of \cite{RX}. The latter says 
the Banach spaces defined by norms 
$\|b\|_{L_q(x^\alpha)}=\| x^\alpha b+ b x^\alpha\|_q$ 
interpolate, so that $L_{\frac p\theta}(x^\theta)=(L_\infty(x^0), L_p(x))_\theta$. As a corollary,
$$\Big\| x^\theta b +b x^\theta \Big\|_{\frac p \theta} \leq C_{\frac p\theta}
\big\|b\big\|_\infty ^{1-\theta}\big\| x b +b x \big\|_{p}^\theta.$$

To avoid the use of \cite{RX} we provide an alternate proof of the
latter inequality with a better constant only when $p=1$ and $\theta\leq \frac 12$. Assuming $\|b\|_\infty\leq 1$, we use  the Jensen's inequality from \cite{BrownKos} for the convex function $x\mapsto x^{\frac 1{2\theta}}$ (for us it follows easily from the operator convexity of $x^\alpha$ for $\alpha\in [1,2]$ and an iteration argument): 
\begin{eqnarray*}
\Big\| x^\theta b +b x^\theta \Big\|_{\frac 1 \theta}^{\frac 1 \theta } &\leq & 
2^{\frac 1 \theta} \Big(\big\| x^\theta b\big\|_{\frac 1 \theta}^{\frac 1\theta} +
\big\| bx^\theta \big\|_{\frac 1 \theta}^{\frac 1\theta}\Big)\\
&\leq& 2^{\frac 1 \theta}\tau \Big( \big(b^*x^{2\theta}b\big)^{\frac 1{2\theta}}+\big(bx^{2\theta}b^*\big)^{\frac 1{2\theta}} \Big)\\
& \leq &2^{\frac 1 \theta}\tau \Big( b^*xb + bxb^*\Big)\\
&\leq & 2^{\frac 1 \theta}\big\|xb+bx \big\|_1.
\end{eqnarray*}
\end{proof}

\begin{lemma}\label{commaz} 
There is an absolute constant $C>0$ and constants $C_t$ ($t>1$) so that :
\begin{itemize}
\item If $q>p\geq 1$, and $x\in L_{ p}(\M)$, $x=x^*$ 
and $b\in \M$ then
\begin{equation}\label{maa}\Big\| \big[M_{p,q} (x),b\Big] \Big\|_{q} \leq C_q
\big\|b\big\|_\infty ^{1-\frac pq}\big\| [x, b]\big\|_{p}^{\frac pq} .\end{equation}
\item If $p>q\geq 1$, and $x\in L_{ p}(\M)$, $x=x^*$ 
and $b\in \M$ then
\begin{equation}\label{maa2}\Big\| \big[M_{p,q} (x),b\Big] \Big\|_{q} \leq C\frac p q
\big\|x\big\|_p ^{\frac pq-1}\big\| [x, b]\big\|_{p} .\end{equation}
\end{itemize}
\end{lemma}

\begin{proof}
For \eqref{maa}, write $e_+=1_{[0,\infty)}(x)$ and $e_-=1_{(-\infty,0)}(x)$ and put 
$b_{\pm,\pm}=e_{\pm} be_{\pm}$. So that 
$$\big[M_{p,q} (x),b\big]= \big[x_+^{\frac  pq},b_{+,+}\big] - \,\big[x_-^{\frac  pq},b_{-,-}\big]+ \,
\big( x_+^{\frac  pq}b_{+,-}+b_{+,-}x_-^{\frac  pq}\big) - \big( x_-^{\frac  pq}b_{-,+}+b_{-,+}x_+^{\frac  pq}\big).$$
We can apply either Lemma \ref{com} or \ref{sum} to each term. In any case, the upper bound we get is smaller than the right side of \eqref{maa}.

A similar argument works for \eqref{maa2}.
\end{proof}

\begin{rk}\label{gen}{\rm
The techniques developed here work if one replaces
$M_{p,q}$ by any function $f:\bR \to \bR$. With such a general function $f$,
$\ref{commaz}$ boils down to the boundedness of some Schur multipliers on 
$S_p[L_p(\M)]$ (by the discretization from \cite{RX}), this is the argument of
\cite{Dav}. This also explains 
why the results of \cite{Dav, AP, PoSu} remain true for semifinite von Neumann algebras.}
\end{rk}

\section{General case}

 In the general case, we use the Haagerup definition of $L_p$-spaces
 \cite{terp} and the Haagerup reduction technique from \cite{HJX}
 (see \cite{CPPR} for extension from states to weights). As the
 construction is very technical, we only give a sketch to keep the paper short.  Let $\M$ be a
 general von Neumann algebra with a fixed faithful normal semifinite
 weight $\phi$ (we use the classical notation $\mathfrak n_\phi$, 
$\mathfrak m_\phi$,... for constructions associated to $\phi$).  As usual $\sigma^\phi$ denotes the automorphisms group
 of $\phi$.  We let $\hat \M=M \rtimes_{\sigma^\phi} \bR$ be the core of
 $\M$. It is a semifinite von Neumann algebra with a distinguished
 trace $\tau$ such that $\tau\circ\hat \sigma_s=e^{-s}\tau$ where
 $\hat \sigma$ is the dual action of $\bR$ on $\hat \M$. The
 definition is then
$$L_p^\phi(\M)= \Big\{ f\in L_0(\hat \M,\tau)\;|\;
 \hat\sigma_s(x)=e^{-\frac sp} x\Big\}.$$ Then $L_1^\phi(M)$ is order
 isometric to $M_*$ and the evaluation at 1 is denoted by ${\rm
   tr}$. The $L_p^\phi$ norm is given by $\|x\|_p^p= {\rm tr} |x|^p$. We
 also denote by $D_\phi$ the Radon-Nykodym derivative of the dual
 weight $\hat \phi$ with respect to $\tau$.

These $L_p^\phi$ spaces are disjoint and the norm topology coincide with
the measure topology of $L_0(\hat M,\tau)$ (Proposition 26 in
\cite{terp}). The construction does not depend on the choice of $\phi$
up to $*$-topological isomorphisms (see below) so that we may drop the superscript $\phi$ when no confusion can arise.

The Haagerup reduction theorem is (see Theorem 2.1 in \cite{HJX} 
or Theorem 7.1 in \cite{CPPR}):

\begin{thm} For any $(\M,\phi)$ there
is a bigger von Neumann algebra $(\R,\tilde \phi)$ where $\tilde \phi$ 
a nfs weight extending $\phi$, a family $a_n$ in the center of the centralizer of $\tilde \phi$ so that
\begin{enumerate}[i)]
\item There is a conditional expectation $\E : \R\to \M$ such that 
$${\phi} \circ \E = \tilde{\phi} \quad \mbox{and} \quad \E \circ \sigma_s^{\tilde{\phi}} = \sigma_s^{{\phi}} \circ \E \quad \mbox{ for all } \quad s \in \mathbb{R}.$$
\item The centralizer $\R_n$ of $\phi_n(.)=\tilde \phi(e^{-a_n}.)$ is
  semifinite for all $n\geq 1$ (with trace $\phi_n$).
\item There exists conditional expectations $\E_n:\R\to \R_n$ such that 
$$\tilde{\phi} \circ \E_n = \tilde{\phi} \quad \mbox{and} \quad \E_n \circ \sigma_s^{\tilde{\phi}} = \sigma_s^{\tilde{\phi}} \circ \E_n \quad \mbox{ for all } \quad s \in \mathbb{R}.$$
\item $\E_n(x) \to x$ $\sigma$-strongly for $x \in
  \mathfrak{n}_{\tilde{\phi}}$ and $\bigcup_{n \geq 1} \R_n$ is
  $\sigma$-strongly dense in $\R$.
\end{enumerate}
\end{thm}
The modular conditions for the conditional expectations imply that 
we can view $L_p(\M)$ and $L_p(\R_n)$ as subspaces of $L_p(\R)$ and there
are extensions:
$$\E^p: L_p(\R) \rightarrow L_p(\M) \quad \mbox{and} \quad 
\E_n^p: L_p(\R) \rightarrow L_p(\R_n).$$
Moreover from {\it iv)}, for any $x\in L_p(\R)$ ($1\leq p<\infty$) we have (see Lemma 7.3 in \cite{CPPR} for instance):
$$\lim_{n\to \infty}\big\| \E_n^p(x) - x\big\|_p =0.$$

Now we make explicit the independence of $L_p(\R_n)$ relative the choice of
the weight. Considering $\R_n$ with $\phi_n$ or $\tilde \phi_n$
gives two constructions, the corresponding 
spaces of measurable operators $N_{\phi_n}=L_0(
\R_n \rtimes_{\sigma^{\phi_n}} \bR, \hat \phi_n)$ and $N_{\tilde \phi}=L_0( \R_n
\rtimes_{\sigma^{\tilde \phi}} \bR, \tau)$ in which the $L_p$-spaces
live.  By Corollary 38 in \cite{terp}, there is a topological
$*$-homomorphism $\kappa :N_{\tilde \phi}\to N_{\phi_n}$ so that $\kappa(L_p^{\tilde
  \phi}(\R_n))=L_p^{\phi_n}(\R_n)$ and is isometric on $L_p$. 

As $\phi_n$ is a trace, we know that $\R_n \rtimes_{\sigma^{\phi_n}}
\simeq \R_n \otimes L_\infty(\bR)$ and the identification $\iota_p :
L_p(\R_n,\phi_n)\to L_p^{\phi_n}(\R_n)$ is $\iota_p(x)=x\otimes
e^{\frac .p}$. Hence we get isometric isomorphisms
$\kappa_p=\iota_p^{-1}\circ \kappa: L_p(\R_n) \to L_p(\R_n,\phi_n)$
that are compatible with left and right multiplications by elements of
$\R_n$ and powers in the sense that for $1\leq q,p< \infty$ and $x\in
L_p^+(\R_n)$
\begin{equation}\label{compow}
\kappa_p(x)^{\frac pq}= \kappa_q\big(x^{\frac pq}\big).\end{equation} 
One can check that $\kappa_p$ is formally given by
$\kappa_p(D^{\frac 1
  {2p}}_{\tilde \phi}xD^{\frac 1 {2p}}_{\tilde \phi})=e^{-\frac
  {a_n}{2p}} x e^{-\frac {a_n}{2p}}$ for  $x\in \mathfrak
m_{\phi_n}$.

Now we can conclude to the proof of the theorem in the general case.
Take $x$ and $y$ in $L_p(M)$, then 
$$ \big\| x-y\big\|_p =\lim_{n\to \infty} \big\| \E_n(x) -\E_n(y)
\big\|_{L_p(\R_n)}= \lim_{n\to \infty} \big\| \kappa_p(\E_n(x))
-\kappa_p(\E_n(y)) \big\|_{L_p(\R_n,\phi_n)}.$$ By Lemma 3.2 in
\cite{Ray}, the map $M_{p,q}$ is continuous on $N_{\tilde \phi}$, thus
  also $L_p\to L_q$, hence
$$ \big\| M_{p,q}(x)-M_{p,q}(y)\big\|_q =\lim_{n\to \infty} \big\| 
\kappa_q(M_{p,q}(\E_n(x)))
-\kappa_q(M_{p,q}(\E_n(y))) \big\|_{L_q(\R_n,\phi_n)}.$$
But thanks to \eqref{compow}, $\kappa_q(M_{p,q}(\E_n(x)))=M_{p,q}(\kappa_p(\E_n(x))$, so that we can use the estimate for semifinite von Neumann algebras to conclude.

\bigskip

In the same way, all inequalities from section 2 can be extended to arbitrary von Neumann algebras (except Remark \ref{gen} as one can not make sense of 
$f(x)\in L_q$ when $x\in L_p^{sa}$ for general functions other than powers).

\bigskip

\textbf{Acknowledgement.} The author would like to thank Masato
Mimura and Gilles Pisier for asking the question on the best H\"older
exponents for the Mazur maps. The author is supported by
ANR-2011-BS01-008-01.

\bibliographystyle{plain}

\begin{thebibliography}{10}

\bibitem{AP}
A.~B. Aleksandrov and V.~V. Peller.
\newblock Functions of operators under perturbations of class {${\bf S}_p$}.
\newblock {\em J. Funct. Anal.}, 258(11):3675--3724, 2010.

\bibitem{Bha}
Rajendra Bhatia.
\newblock {\em Matrix analysis}, volume 169 of {\em Graduate Texts in
  Mathematics}.
\newblock Springer-Verlag, New York, 1997.

\bibitem{BrownKos}
Lawrence~G. Brown and Hideki Kosaki.
\newblock Jensen's inequality in semi-finite von {N}eumann algebras.
\newblock {\em J. Operator Theory}, 23(1):3--19, 1990.

\bibitem{CPPR}
Martijn Capers, Javier Parcet, Mathilde Perrin, and \'Eric Ricard.
\newblock Noncommutative {D}e {L}eeuw theorems.
\newblock {\em Preprint arXiv:1407.2449}, 2014.

\bibitem{CL}
Eric~A. Carlen and Elliott~H. Lieb.
\newblock Optimal hypercontractivity for {F}ermi fields and related
  noncommutative integration inequalities.
\newblock {\em Comm. Math. Phys.}, 155(1):27--46, 1993.

\bibitem{Dav}
E.~B. Davies.
\newblock Lipschitz continuity of functions of operators in the {S}chatten
  classes.
\newblock {\em J. London Math. Soc. (2)}, 37(1):148--157, 1988.

\bibitem{HJX}
Uffe Haagerup, Marius Junge, and Quanhua Xu.
\newblock A reduction method for noncommutative {$L_p$}-spaces and
  applications.
\newblock {\em Trans. Amer. Math. Soc.}, 362(4):2125--2165, 2010.

\bibitem{PX}
Gilles Pisier and Quanhua Xu.
\newblock Non-commutative {$L^p$}-spaces.
\newblock In {\em Handbook of the geometry of {B}anach spaces, {V}ol.\ 2},
  pages 1459--1517. North-Holland, Amsterdam, 2003.

\bibitem{PoSu}
Denis Potapov and Fedor Sukochev.
\newblock Operator-{L}ipschitz functions in {S}chatten-von {N}eumann classes.
\newblock {\em Acta Math.}, 207(2):375--389, 2011.

\bibitem{Ray}
Yves Raynaud.
\newblock On ultrapowers of non commutative {$L_p$} spaces.
\newblock {\em J. Operator Theory}, 48(1):41--68, 2002.

\bibitem{RX}
{\'E}ric Ricard and Quanhua Xu.
\newblock Complex interpolation of weighted noncommutative {$L_p$}-spaces.
\newblock {\em Houston J. Math.}, 37(4):1165--1179, 2011.

\bibitem{terp}
Marianne Terp.
\newblock {L}$^p$ spaces associated with von neumann algebras.
\newblock {\em Notes, K\o benhavns Universitets Matematiske Institut}, 1981.

\end{thebibliography}

\end{document}